\documentclass[11pt]{amsart}

\usepackage{amscd}
\usepackage{amsfonts}
\usepackage{ascmac}
\usepackage[usenames]{color}
\usepackage{graphics}
\usepackage{amssymb,latexsym}
\usepackage{amsmath,amsthm}

\theoremstyle{plain}
\newtheorem{thm}{Theorem}[section]
\newtheorem{prop}[thm]{Proposition}
\newtheorem{lem}[thm]{Lemma}
\newtheorem{cor}[thm]{Corollary}

\newtheorem{ex}[thm]{Example}
\newtheorem{rem}[thm]{Remark}

\newcommand{\bfC}{{\mathbf C}}
\newcommand{\bfP}{{\mathbf P}}
\newcommand{\bfR}{{\mathbf R}}

\newcommand{\barf}{{\overline f}}

\newcommand{\bari}{{\overline i}}
\newcommand{\barj}{{\overline j}}
\newcommand{\bark}{{\overline k}}
\newcommand{\barl}{{\overline \ell}}

\newcommand{\baru}{{\overline u}}
\newcommand{\barp}{{\overline p}}

\newcommand{\barz}{{\overline z}}

\newcommand{\barbet}{{\overline \beta}}

\newcommand{\barmu}{{\overline \mu}}

\newcommand{\barpartial}{{\overline \partial}}

\newcommand{\mapright}[1]{\smash{\mathop{   \hbox to 0.7cm{\rightarrowfill}}
  \limits^{#1}}}

\def\grad{\mathrm{grad}}

\renewcommand{\emph}[1]{{\color{red} \it #1}}

\definecolor{orange}{cmyk}{0, 0.7, 1, 0}
\definecolor{light-green}{cmyk}{0.5, 0, 0.5, 0}
\definecolor{light-blue}{cmyk}{0.5, 0, 0, 0}
\definecolor{light-yellow}{cmyk}{0,0,0.6,0}
\definecolor{dark-green}{cmyk}{0.7, 0, 0.7, 0.5}

\title{Cahen--Gutt moment map, closed Fedosov star product and structure of the automorphism group}
\author{Akito Futaki and Hajime Ono}
\address{Yau Mathematical Sciences Center, Tsinghua University, Haidian district, Beijing 100084, China}
\email{akito@math.tsinghua.edu.cn}
\address{Department of Mathematics, Saitama University, 255 Shimo-Okubo, Sakura-Ku,
Saitama 380-8570, Japan}
\email{hono@rimath.saitama-u.ac.jp}

\date{October 28, 2018}

\begin{document}
\begin{abstract}
We show that if a compact K\"ahler manifold $M$ with non-negative Ricci curvature admits closed Fedosov star product then 
the reduced Lie algebra of holomorphic vector fields on $M$ is reductive.
This comes in pair with the obstruction previously found by La Fuente-Gravy \cite{La Fuente-Gravy 2016_2}. 
More generally we consider the squared norm of Cahen--Gutt moment map
as in the same spirit of Calabi functional for the scalar curvature in cscK problem, and 
prove a Cahen--Gutt version
of Calabi's theorem on the structure of the Lie algebra of holomorphic vector fields
for extremal K\"ahler manifolds. 
\end{abstract}

\maketitle

%%%%%%%%%%%%%%%%%%%%%%%
%%%%%%%%%%%%%%%%%%%%%%%

\section{Introduction.}
A deformation quantization is a formal associative deformation of a Poisson algebra $(C^\infty(M), \cdot, \{\cdot,\cdot\})$
into the space $C^\infty(M)[[\nu]]$ of formal power series in $\nu$ with a composition law $\ast$ called the star product
with the following property. The constant function $1$ is a unit, and if 
we write for $f,\ g  \in C^\infty(M)$ 
\begin{eqnarray}\label{star}
f \ast g = \sum_{r=0}^\infty C_r(f,g) \nu^r,
\end{eqnarray}
then $\ast$ is required to satisfy
\begin{eqnarray}\label{star2}
C_0(f,g) = f\cdot g, \qquad C_1(f,g) - C_1(g,f)  = \{f,g\},
\end{eqnarray}
and $C_r$'s are required to be bidifferential operators.
For symplectic manifolds, the existence of star products was shown by Dewilde and Lecompte \cite{DWL}, Fedosov 
\cite{Fedosov1994} and Omori, Maeda and Yoshioka \cite{OMY1991}. For general Poisson manifolds, the existence
of star products was shown by Kontsevich \cite{Kon2003}. 
A star product on a compact symplectic manifold $(M,\omega)$ of dimension $2m$ 
is called closed  (in the sense of Connes-Flato-Sternheimer \cite{CFS}) if 
\begin{eqnarray}\label{star3}
\int_M F\ast H\ \omega^m = \int_M H \ast F\ \omega^m
\end{eqnarray}
for all $F,\ H \in C^\infty(M)[[\nu]]$.

Let $(M, \omega)$ be a compact connected symplectic manifold of dimension $2m$.
In \cite{CahenGutt}, Cahen and Gutt defined a moment map $\mu$ on the space of symplectic connections for
the action of the group of Hamiltonian diffeomorphisms. In \cite{La Fuente-Gravy 2016}, La Fuente-Gravy
showed that if the Fedosov star product $\ast_\nabla$ is closed for a symplectic connection $\nabla$ then $\mu(\nabla)$ is constant. Assuming $M$ is K\"ahler and
fixing a K\"ahler class, La Fuente-Gravy further 
defined in \cite{La Fuente-Gravy 2016_2} a Lie algebra character $\mathrm{Fut} : \mathfrak g \to \mathbf R$
where $\mathfrak g$  is 
the reduced Lie algebra 
of holomorphic vector fields on $M$,
and showed that 
if there exists a K\"ahler metric in the fixed K\"ahler class such that 
$\mu(\nabla)$ is constant for the Levi-Civita connection $\nabla$ 
then the character $\mathrm{Fut}$ vanishes. In particular, non-vanishing of the character $\mathrm{Fut}$
obstructs the existence of a K\"ahler metric in the fixed K\"ahler class such that the Fedosov star product $\ast_\nabla$ 
is closed.  
(Recall that, by 
definition, the reduced Lie algebra $\mathfrak g$ of holomorphic vector fields on a compact K\"ahler manifold
consists of holomorphic vector fields of the form $\grad^\prime f$ for some complex valued smooth
function $f$ (c.f. \cite{GauduchonLN}). In fact, $\mathfrak g$ does not depend on the choice of the K\"ahler metric.)
La Fuente-Gravy showed that, when the fixed K\"ahler class is integral and equal to $c_1(L)$ for some ample line bundle
 $L$, the character he 
 defined is one of the obstructions for the polarized manifold $(M,L)$ 
to be asymptotically Chow semistable obtained in \cite{futaki04-1}, see also \cite{futakiICCM10}
 for more applications.
 
 In the problem of finding constant scalar curvature K\"ahler (cscK) metrics, Fujiki \cite{fujiki92} and 
 Donaldson \cite{donaldson97} set up a moment map $\tau$ on the space of complex structures compatible with a fixed
 symplectic form $\omega$ where $\tau(J)$ at a complex structure $J$ is the scalar curvature of the K\"ahler manifold
 $(M, \omega, J)$. In this cscK problem, we also have the Lie algebra character which obstructs the existence of
 cscK metrics (\cite{futaki83.1}). On the other hand, we also have another obstruction which claims that the Lie algebra of 
 all holomorphic vector fields of a cscK manifold has to be reductive (\cite{matsushima57}, \cite{Lic}). 
 This is further extended by Calabi \cite{calabi85} to a structure theorem of the Lie algebra for compact extremal K\"ahler
 manifolds. As can be seen in other similar problems (see e.g. \cite{futaki07.1}, \cite{FO_reductive17}, \cite{Lahdili17}, \cite{AGGarcia16}), 
 the two obstructions of the Lie algebra character and the reductiveness
 come always in pair. The purpose of this paper is to show that this is the case, namely the following is the main result of this paper.
 
 \begin{thm}\label{reductiveness}
 Let $M$ be a compact K\"ahler manifold. If there exists a K\"ahler metric with non-negative Ricci curvature such that $\mu(\nabla)$ is constant for the 
 Cahen--Gutt moment map $\mu$ and the Levi-Civita connection $\nabla$ then
 the reduced Lie algebra $\mathfrak g$ of holomorphic vector fields is reductive. In particular, if $\mathfrak g$
 is not reductive then there is no K\"ahler metric with non-negative Ricci curvature such that the Fedosov star product $\ast_\nabla$ for the Levi-Civita
 connection $\nabla$ is closed.
\end{thm}
To show this we define Cahen--Gutt version of extremal K\"ahler metrics and prove a similar structure theorem as 
the Calabi extremal K\"ahler metrics.
 The strategy of the proof of the structure theorem for Cahen--Gutt extremal K\"ahler manifold is to use the formal 
finite dimensional argument for the Hessian formula of the squared norm of the moment map
given by Wang \cite{Lijing06}. The merit of Wang's argument is that once the suitable modification of the 
Lichnerowicz operator is made we can apply his formal argument without using the explicit expression of
the modified Lichnerowicz operator. This strategy has been used previously for perturbed extremal K\"ahler metrics
in \cite{futaki07.1} and for conformally K\"ahler, Einstein-Maxwell metrics in \cite{FO_reductive17}. 

This paper is organized as follows. In section 2 we recall Cahen--Gutt moment map. We show in K\"ahler situation 
an explicit expression
of the Lie derivative of the connection by a Hamiltonian vector field (Lemma \ref{main lemma}). 
From this lemma we see
that the Lie derivative of the connection by a Hamiltonian vector field vanishes if and only if the Hamiltonian vector
field is a holomorphic Killing vector field. Using the Cahen--Gutt moment map formula we reprove the result of
La Fuente-Gravy that the Lie algebra character is independent of the choice of K\"ahler metric in the fixed K\"ahler class.
But our set-up is $\omega$-fixed and $J$-varying, and thus what we prove is independence of the choice of $J$.
In section 3 we give an alternate proof of Lemma \ref{main lemma}. 
The computations in this section are used in section 4. 
In section 4 we apply Wang's formal argument to prove the structure theorem for Cahen--Gutt extremal K\"ahler manifolds.
As we will only use the fact that closedness of the Fedosov star product $\ast_\nabla$ 
implies that $\mu(\nabla)$ is constant, 
we will not reproduce the detailed account on closedness of Fedosov star product. We expect an 
interested reader will refer 
to La Fuente-Gravy's articles \cite{La Fuente-Gravy 2016}, \cite{La Fuente-Gravy 2016_2} for it.

\vspace{0.2cm}

\noindent
Acknowledgement: The authors are grateful to the referee for careful reading and pointing out the necessity of the curvature condition for the proof of reductiveness.
%%%%%
\section{Cahen--Gutt moment map.}
Let $(M,\omega)$ be a symplectic manifold of dimension $2m$. 
A symplectic connection $\nabla$ is a torsion free affine connection such that $\nabla\omega = 0$.
There always exists a symplectic connection on any symplectic manifold, see e.g. \cite{BCGRS}, 
section 2.1. Unlike the Levi-Civita connection on a Riemannian manifold, a symplectic connection 
is not unique on a symplectic manifold. Given two symplectic connections $\nabla$ and $\nabla^\prime$, we write the difference by $S$:
$$\nabla_X Y - \nabla^\prime_X Y = S(X,Y).$$
Then $\omega(S(X,Y),Z)$ is totally symmetric in $X$, $Y$ and $Z$. Conversely, if $\nabla$ is a symplectic connection
and $\omega(S(X,Y),Z)$ is totally symmetric, then $\nabla^\prime := \nabla + S$ is a 
symplectic connection, see \cite{BCGRS}, 
section 2.1. 
In the geometry of symmplectic connections, $\omega = \omega_{ij}dx^i \wedge dx^j$ and $(\omega^{ij}) = (\omega_{ij})^{-1}$ are used to raise and lower
the indices, and we write
\begin{equation}
\underline{S} = \underline{S}_{ijk}\,dx^i \otimes dx^j \otimes dx^k 
\end{equation}
for
\begin{equation}
S = S_{ij}{}^k\, dx^i \otimes dx^j \otimes \frac{\partial}{\partial x^k}
\end{equation}
with $\underline{S}_{ijk} = S_{ij}{}^\ell\omega_{\ell k}$. With this notation, $\underline{S}_{ijk}$ is symmetric in
$i$, $j$ and $k$.
Thus, on a symplectic manifold $(M,\omega)$, the space of
symplectic connections, denoted by $\mathcal E (M,\omega)$, is an affine space modeled on the set of all smooth sections
$\Gamma(S^3(T^\ast M))$ of symmetric covariant 3-tensors. Thus we may identify $\mathcal E (M,\omega)$ as 
$$ \mathcal E (M,\omega) \cong \nabla + \Gamma(S^3(T^\ast M)).$$

From now on we assume $M$ is a closed manifold. 
On $\mathcal E (M,\omega)$ there is a natural symplectic structure $\Omega^{\mathcal E}$
defined at $\nabla$ given by
\begin{equation}\label{symp str}
 \Omega^{\mathcal E}_\nabla(\underline{A},\underline{B}) = 
\int_M \omega^{i_1j_1}\omega^{i_2j_2}\omega^{i_3j_3}
\underline{A}_{i_1i_2i_3}\,\underline{B}_{j_1j_2j_3}\ \omega_m
\end{equation}
for $\underline{A},\ \underline{B} \in T_\nabla \mathcal E(M,\omega) \cong \Gamma(S^3(T^\ast M))$
where $\omega_m := \frac{\omega^m}{m!}$. Since
$\Omega^{\mathcal E}_\nabla$ is independent of $\nabla$ we may omit $\nabla$ and write
$\Omega^{\mathcal E}$. 
There is a natural action of the group of symplectomorphisms (i.e. symplectic diffeomorphisms) 
of $(M,\omega)$ on
$\mathcal E(M,\omega)$, which is given for a symplectomorphism $\varphi$ by
$$ (\varphi(\nabla))_X Y = \varphi_\ast (\nabla_{\varphi_\ast^{-1} X} \varphi_\ast^{-1} Y)$$
for any $\nabla \in \mathcal E(M,\omega)$ and any smooth vector fields  $X$ and $Y$ on $M$.
This action preserves the symplectic structure $\Omega^{\mathcal E}$ 
on $\mathcal E(M,\omega)$. 
In particular, the group $\mathrm{Ham}(M,\omega)$ of Hamiltonian diffeomorphisms acts on
$\mathcal E(M,\omega)$ as symplectomorphisms. 
Let $X_f$ be a Hamiltonian vector field on $M$ for a smooth function $f$ on $M$, that is
\begin{equation}\label{Hamiltonian}
 i(X_f) \omega = df.
 \end{equation}
Then the induced infinitesimal action of $- X_f$ on $\mathcal E(M,\omega)$ is computed as
\begin{eqnarray}\label{tangent}
(L_{X_f}\nabla)_Y Z
&=& [X_f,\nabla_Y Z] - \nabla_{[X_f, Y]} Z - \nabla_Y [X_f,Z]\nonumber\\
&=& R^\nabla(X_f,Y)Z + (\nabla\nabla X_f )(Y,Z)
\end{eqnarray}
where $R^\nabla$ is the curvature tensor of $\nabla$, i.e. $R^\nabla(X,Y)Z = \nabla_X\nabla_Y Z - \nabla_Y\nabla_X Z - \nabla_{[X,Y]} Z$, and
$(\nabla\nabla X )(Y,Z) = \nabla_Y\nabla_Z X - \nabla_{\nabla_Y Z} X$. 
For $\nabla \in \mathcal E(M,\omega)$ we put
\begin{eqnarray}\label{moment map}
\mu(\nabla) &=& \nabla_p\nabla_q \mathrm{Ric}(\nabla)^{pq} \nonumber\\
&& \quad - \frac12 \mathrm{Ric}(\nabla)_{pq}\mathrm{Ric}(\nabla)^{pq}
+ \frac14 \mathrm{R}(\nabla,\omega)_{pqrs}\mathrm{R}(\nabla,\omega)^{pqrs}
\end{eqnarray}
where 
\begin{equation}\label{curvature}
\mathrm{R}(\nabla,\omega)(X,Y,Z,W) = \omega (R(X,Y)Z, W)
\end{equation}
and 
$$ \mathrm{Ric}(X,Y) = - \mathrm{tr} ( Z \mapsto R(X,Z)Y).$$
\begin{thm}[Cahen--Gutt \cite{CahenGutt}]\label{CahenGutt}
The functional $\mu$ on $\mathcal E(M,\omega)$ gives a
moment map
for the action of $\mathrm{Ham}(M,\omega)$.
\end{thm}
This follows from the formula
\begin{eqnarray}\label{moment2}
\left.\frac{d}{dt}\right|_{t=0} \int_M \mu(\nabla + tA)\,f\,\omega_m =
\Omega^{\mathcal E}(\underline{L_{X_f}\nabla}, \underline{A}).
\end{eqnarray}
Note from \eqref{tangent} that
\begin{equation}\label{Cahen-Gutt2}
\underline{L_X\nabla} = (X^sR(\nabla,\omega)_{squt} + \nabla_q\nabla_u X^s\,\omega_{st})\,dx^q \otimes dx^u \otimes dx^t.
\end{equation}

Now we assume that $M$ is a compact K\"ahler manifold and that $\omega$ is a fixed symplectic form. We set
$$ \mathcal J(M,\omega) = \{ J\ \text{integrable complex\ structure}\ |\ (M,\omega, J)\ \text{is\ a\ K\"ahler\ manifold.}\}$$
La Fuente-Gravy \cite{La Fuente-Gravy 2016}, \cite{La Fuente-Gravy 2016_2} considered the {\it Levi-Civita map}
$lv : \mathcal{J}(M,\omega) \to \mathcal E(M,\omega)$ sending $J$ to the Levi-Civita connection $\nabla^J$
of the K\"ahler manifold $(M,\omega,J)$. 
The following is a key lemma to this paper.
\begin{lem}\label{main lemma}
If we choose local holomorphic coordinates $z^1, \cdots, z^m$ then
for any smooth function $f$ we have
\begin{eqnarray}\label{infinitesimal}
\underline{L_{X_f}\nabla^J} &=& 
 f_{ijk} dz^i \otimes dz^j \otimes dz^k + f_{\bari\barj\bark} dz^\bari \otimes dz^\barj \otimes dz^\bark \\
&& + f_{ij\bark} dz^i \otimes dz^j \otimes dz^\bark + f_{\bari\barj k} dz^\bari \otimes dz^\barj \otimes dz^k \nonumber\\
&& + f_{ik\barj} dz^i \otimes dz^\barj \otimes dz^k + f_{\bari\bark j} dz^\bari \otimes dz^j \otimes dz^\bark \nonumber\\
&& + f_{jk\bari} dz^\bari \otimes dz^j \otimes dz^k + f_{\barj\bark i} dz^i \otimes dz^\barj \otimes dz^\bark \nonumber
\end{eqnarray}
where the lower indices of $f$ stand for the covariant derivatives, e.g. $f_{ij\bark} = \nabla_\bark \nabla_j \nabla_i f$.
\end{lem}
\begin{proof}
Write $\partial_i = \partial/\partial z^i$, 
$\barpartial_i = \partial/\partial \overline{z^i}$ for short. We use the standard tensor calculus notations in
K\"ahler geometry. Thus, $R_{ABCD}$ denotes the K\"ahlerian curvature tensor with $A,\ B,\ C, D$ running from $1, \cdots, m, \bar{1}, \cdots, \bar{m}$ and we use the metric tensor or its inverse to lower and raise indices.
First of all, from \eqref{Hamiltonian} we have 
\begin{equation}\label{X_f}
X_f = -J\mathrm{grad}f = -\sqrt{-1} f^\ell \partial_\ell + \sqrt{-1} f^\barl \barpartial_\ell.
\end{equation}
Secondly, from \eqref{curvature} we have
\begin{equation}\label{curvature2}
R(\nabla,\omega)_{ABi\barj} = -\sqrt{-1} R_{ABi\barj}.
\end{equation}
From \eqref{Cahen-Gutt2}, \eqref{X_f} and \eqref{curvature2} we obtain
\begin{equation}\label{Cahen-Gutt3}
\underline{L_X\nabla}(\partial_i,\partial_j, \partial_k) = f_{ijk}
\end{equation}
since $R_{ijCD} = 0$ on K\"ahler manifolds, and
\begin{eqnarray}
\underline{L_X\nabla}(\partial_i,\partial_j, \partial_\bark) &=& - R_{\barl i\bark j} f^\barl + \nabla_i\nabla_j  f_{\bark} \label{Cahen-Gutt4}\\
&=&  - R_{\bark i\barl j} f^\barl + \nabla_i\nabla_{\bark}  f_j \nonumber\\
&=& \nabla_\bark\nabla_i f_j = f_{ji\bark} = f_{ij\bark}. \nonumber
\end{eqnarray}
Since we know $\underline{L_X\nabla}$ is totally symmetric, taking the complex conjugates of \eqref{Cahen-Gutt3} and \eqref{Cahen-Gutt4} we obtain
\eqref{infinitesimal}.

Alternatively, one may compute
\begin{eqnarray*}
&&R(-\sqrt{-1} f^\ell \partial_\ell + \sqrt{-1} f^\barl \barpartial_\ell,\partial_i + \barpartial_i)(\partial_j + \barpartial_j)\\
&& \quad + \nabla_{\partial_i + \barpartial_i}\,\nabla_{\partial_j + \barpartial_j} (-\sqrt{-1} f^\ell \partial_\ell + \sqrt{-1} f^\barl \barpartial_\ell),
\end{eqnarray*}
from which one obtains \eqref{infinitesimal}, and the result shows that $\underline{L_X\nabla}$ is totally symmetric.
\end{proof}
An alternate proof of Lemma \ref{main lemma} is given in the next section, see \eqref{nablader2},
where we use an explicit description of the differential $lv_\ast : T_J \mathcal J (M,\omega) \to
T_{\nabla^J}\mathcal E(M,\omega)$, see Lemma \ref{S^3}.

If $lv^\ast \Omega^{\mathcal E}$ is non-degenerate, $lv^\ast\mu$ gives the moment map for the action of
$\mathrm{Ham}(M,\omega)$ with respect to the symplectic structure $lv^\ast \Omega^{\mathcal E}$.
However the nondegeneracy of $lv^\ast \Omega^{\mathcal E}$ is not obvious, and in \cite{La Fuente-Gravy 2016}, Proposition 17, a sufficient condition for the non-degeneracy of $lv^\ast \Omega^{\mathcal E}$ is given (see also Lemma \ref{Ric} and Remark \ref{Prop 17} of the present paper).
Disregarding this difficulty, 
notice that this symplectic structure is different from the one used in Donaldson \cite{donaldson97} and
Fujiki \cite{fujiki92}. As we noted in the introduction, by choosing different symplectic structures on
$\mathcal J(M,\omega)$, we obtain similar results for other nonlinear geometric problems just as the cscK problem, see \cite{futaki06}, \cite{futaki07.1}, 
\cite{FO_reductive17}. Each case of them should be studied in terms K-stability.

\begin{prop}\label{nondegenerate}For a real smooth function $f$, 
$L_{X_f}\nabla^J = 0$ if and only if $L_{X_f}J = 0$. In this case, $X_f$ is a holomorphic Killing vector field.
\end{prop}
\begin{proof} We write $\nabla$ instead of $\nabla^J$ for notational simplicity. 
By Lemma \ref{main lemma}, $L_{X_f}\nabla = 0$ implies 
$\nabla^\prime\nabla^{\prime\prime}\nabla^{\prime\prime}f = 0$. By integration by parts, this shows
$$ \int_M |\nabla^{\prime\prime}\nabla^{\prime\prime}f|^2 \omega_m = 0.$$
 This implies $X_f$ is holomorphic, that is, $L_{X_f}J = 0$ since
for a smooth vector field $X = X^{\prime} + X^{\prime\prime}$ we have
\begin{equation}\label{LieJ}
L_X J = 2\sqrt{-1} \nabla^{\prime\prime} X^{\prime}
 - 2\sqrt{-1} \nabla^{\prime} X^{\prime\prime},
\end{equation}
as shown in Lemma 2.3 in \cite{futaki06}.
Lemma \ref{main lemma} also shows, conversely if $L_{X_f}J = 0$ then $L_{X_f}\nabla = 0$. 
In this case, since $X_f$ preserves $\omega$ it is a Killing vector field. 
This completes the proof of Proposition \ref{nondegenerate}.
\end{proof}
Lemma \ref{main lemma} gives an alternate proof of the following result of La Fuente-Gravy.
We consider $\mathfrak g_{\mathbf R}$ consisting of $\grad^\prime f \in \mathfrak g$
of some real smooth function.
\begin{cor}[La Fuente-Gravy \cite{La Fuente-Gravy 2016_2}]\label{Futaki}
Let $(M, \omega)$ be a compact K\"ahler manifold, and $\mathfrak g_{\mathbf R}$ be the real reduced Lie algebra of holomorphic vector fields. We normalize the Hamiltonian functions $f$ so that $\int_M f\ \omega_m = 0$.
Then 
$$ \mathrm{Fut}(\grad^\prime f) := \int_M \mu(\nabla^J)\, f\ \omega_m $$
is independent of the choice of $J \in \mathcal J(M,\omega)$.
\end{cor}
\begin{proof}By the moment map formula \eqref{moment2}, the derivative of $\mathrm{Fut}(\grad^\prime f)$ vanishes
when $L_{X_f}\nabla = 0$. But by Proposition \ref{nondegenerate}, $L_{X_f}\nabla = 0$ is equivalent to $L_{X_f}J = 0$.
This completes the proof of Corollary \ref{Futaki}
\end{proof}
La Fuente-Gravy \cite{La Fuente-Gravy 2016_2} shows that this is an obstruction to the existence of an integrable
complex structure $J$ such that the Levi-Civita connection $\nabla^J$
gives rise to closed Fedosov star product. 
This follows from his observation that the closedness of Fedosov star product implies $\mu(\nabla^J)$ is constant.
He also observed the invariant $ \mathrm{Fut}(\grad^\prime f)$ in Corollary \ref{Futaki} is 
one of the invariants considered in \cite{futaki04-1}. The latter family of invariants includes the standard obstruction
to the existence of K\"ahler-Einstein metrics \cite{futaki83.1}.
In the last section of this paper we will obtain another obstruction for the Levi-Civita connection $\nabla^J$
to give rise to closed Fedosov star product. 

%%%%%%%%%%%
%%%%%%%%%%%
%%%%%%%%%%%
\section{An alternate proof of Lemma \ref{main lemma}.}
In this section we give an alternate proof of Lemma \ref{main lemma}. The results in this section are
used in the next section.
Let $(M,\omega)$ be a compact symplectic manifold of dimension $2m$ with a fixed
symplectic form $\omega$.
Let  $\mathcal J(M,\omega)$ be the set of all $\omega$-compatible integrable complex structures
where $J$ is said to be $\omega$-compatible if $\omega(JX,JY) = \omega(X,Y)$ 
for all vector fields $X$ and $Y$ and $\omega(X,JX) > 0$ for all non-zero $X$.
Thus, for each $J \in \mathcal J(M,\omega)$, the triple $(M, \omega, J)$ determines a K\"ahler
structure. 

Consider $J$ as acting on the cotangent bundle
and decompose the complexified cotangent bundle into holomorphic and anti-holomorphic parts,
i.e. $\pm \sqrt{-1}$-eigenspaces of $J$:
\begin{equation}
T^{\ast}M\otimes\bfC = T_J^{\ast\prime}M \oplus T_J^{\ast\prime\prime}M, 
\qquad T_J^{\ast\prime\prime}M = \overline{T_J^{\ast\prime}M}.
\end{equation}
Take arbitrary $J^{\prime} \in \mathcal J(M,\omega)$, then we also have the decomposition with respect to $J^{\prime}$
\begin{equation}
T^{\ast}M\otimes\bfC = T_{J^{\prime}}^{\ast\prime}M \oplus T_{J^{\prime}}^{\ast\prime\prime}M, 
\qquad T_{J^{\prime}}^{\ast\prime\prime}M = \overline{T_{J^{\prime}}^{\ast\prime}M}.
\end{equation}
If $J^{\prime}$ is sufficiently close to $J$ then $T_{J^{\prime}}^{\ast\prime}M $ 
can be expressed as a graph over $ T_J^{\ast\prime}M$ in the form
\begin{equation}
T_{J^{\prime}}^{\ast\prime}M = \{\ \alpha + \mu(\alpha) \ |\ \alpha \in T_J^{\ast\prime}M\ \}
\end{equation}
for some endomorphism $\mu$ of $T_J^{\ast\prime}M$ into $T_J^{\ast\prime\prime}M$.
 We use the identification of $T_J^{\ast\prime\prime}M$ with $T_J^{\prime}M$ by
 the K\"ahler metric defined by the pair $(\omega, J)$, 
 and then $\mu$ is regarded as
\begin{eqnarray}
 \mu &\in& \Gamma (\mathrm{End}(T_J^{\ast\prime}M, T_J^{\ast\prime\prime}M)) \nonumber \\
 &\cong&  \Gamma (T_J^{\prime}M \otimes T_J^{\ast\prime\prime}M) \cong \Gamma(T_J^{\prime}M\otimes T_J^{\prime}M).
\end{eqnarray}
In the tensor calculus notations this is expressed as
$$ \mu^i{}_{\bark} \mapsto g^{j\bark}\mu^i{}_{\bark}=: \mu^{ij} $$
where we chose a local holomorphic coordinate system 
$(z^1, \cdots, z^m)$ with respect to $J$
and wrote $\omega$ as $\omega = \sqrt{-1}\ g_{i\barj} dz^i \wedge d\overline{z^j}$.
The following is known:
\begin{equation}\label{symmetry}
\mu^{ij} = \mu^{ji},
\end{equation}
see the proof of Lemma 2.1 in \cite{futaki06}.

If $J_t$ is a smooth curve in $\mathcal J(M,\omega)$ with $J_0 = J$ and $\mu(t)$ is the curve in 
$\Gamma (\mathrm{End}(T_J^{\ast\prime}M, T_J^{\ast\prime\prime}M))$
satisfying
$$J_t (\alpha + \mu(t)\alpha) = \sqrt{-1}(\alpha + \mu(t)\alpha)$$ 
with $\dot{\mu}(0) = \mu$, then we have
\begin{equation}\label{J_derivative}
\dot{J}|_{t=0} = 
2\sqrt{-1}\mu - 2\sqrt{-1}\overline{\mu},
\end{equation}
see \cite{futaki06}, pp.353--354.
Let $g_t$ be the K\"ahler metric of $(\omega, J_t)$, i.e. $g_t(\cdot,\cdot) = \omega(\cdot, J_t\cdot)$.
We write the Christoffel symbol in terms of real coordinates
 $$ \Gamma_{t,kj}^i = \frac12 g_t^{i\ell}\left(
\frac{\partial g_{t,\ell j}}{\partial x^k} + \frac{\partial g_{t,\ell k}}{\partial x^j}
- \frac{\partial g_{t,j k}}{\partial x^{\ell}} \right), $$
take the derivative
at $t=0$ and express it in terms of normal coordinates to obtain
the derivative of the covariant derivative
\begin{equation}\label{connection_der}
\dot{\nabla}^i_{kj}dx^k = \dot{\Gamma}_{kj}^i dx^k = \frac12 (\dot g^i{}_{j,k} + \dot g^i{}_{k,j} - \dot g_{jk,}{}^i)dx^k
\end{equation}
where $i$ and $j$ are respectively row and column indices.
Note $g_{t\,ij} = \omega_{ik} J_t^k{}_j$ and that $\omega$ is fixed. Thus we have
$\dot g_{ij} = \omega_{ik} \dot J^k{}_j$.

In local complex coordinates $z^1, \cdots, z^m, \barz^1, \cdots, \barz^m$ giving complex struture $J$
we have
$\omega = \sqrt{-1} g_{i\barj} dz^i \wedge dz^\barj$, $\omega_{\barj i} = - \omega_{i \barj} = - \sqrt{-1} g_{i\barj}$.
It follows from \eqref{J_derivative} that
\begin{equation}\label{gbar_derivative}
\dot{g}_{\bari\,\barj} = \omega_{\bari k} 2\sqrt{-1}\mu^k{}_\barj = 2 \mu_{\bari\,\barj},
\end{equation}
\begin{equation}\label{g_derivative}
\dot{g}_{ij} = \omega_{i \bark} (-2\sqrt{-1})\barmu^\bark{}_j = 2 \barmu_{ij},
\end{equation}
\begin{equation}\label{other_gderivative}
\dot{g}_{\bari\,j} = 0,\ \ \dot{g}_{i\barj} = 0.
\end{equation}

The derivative $\dot{\nabla}$ of the connection is the section of $\mathrm{End}(TM)\otimes T^\ast M \cong T^\ast M \otimes T^\ast M\otimes TM$.
Using $\omega$ a section $A$ of $\mathrm{End}(TM)\otimes T^\ast M \cong T^\ast M \otimes T^\ast M\otimes TM$
can be identified with a section $\underline{A}$ of $\otimes^3 T^\ast M$.
For example, this means
$$ \underline{A}_{\bari jk} =   A^p{}_{jk} \omega_{p\bari}= \sqrt{-1}g_{p\bari} A^p{}_{jk}, \quad
\underline{A}_{ijk} =  A^\barp{}_{jk}\omega_{\barp i}  = -\sqrt{-1}g_{i\barp} A^\barp{}_{jk}.$$

\begin{lem}\label{S^3}
With the identification above, we have
\begin{eqnarray}\label{nablader}
-\sqrt{-1}\underline{\dot{\nabla}} 
&=& \barmu_{ij,k} dz^i\otimes dz^j \otimes dz^k - \mu_{\bari\barj,\bark}dz^\bari\otimes dz^\barj \otimes dz^\bark \\
&&+ \barmu_{ij,\bark} dz^i \otimes dz^j \otimes dz^\bark 
 - \mu_{\bari\barj,k}dz^\bari\otimes dz^\barj \otimes dz^k  \nonumber\\
&& + \barmu_{ik,\barj} dz^i \otimes dz^\barj \otimes dz^k - \mu_{\bari\bark,j} dz^\bari \otimes dz^j \otimes dz^\bark \nonumber\\
&&  + \barmu_{jk,\bari} dz^\bari \otimes dz^j \otimes dz^k  
 - \mu_{\barj\bark,i} dz^i \otimes dz^\barj \otimes dz^\bark   \nonumber
\end{eqnarray}
where the comma $,$ denotes the covariant derivative. In particular, 
 $\underline{\dot{\nabla}}$ is totally symmetric in terms of the three components.
\end{lem}
\begin{proof}
We lower the upper indices of \eqref{connection_der} using $\omega$, and then 
use \eqref{gbar_derivative}, \eqref{g_derivative}
and \eqref{other_gderivative} to obtain
\begin{eqnarray}\label{nablader1}
-\sqrt{-1}\underline{\dot{\nabla}} &=& -\sqrt{-1}\underline{\dot{\Gamma}}\\
&=& \barmu_{ij,k} dz^i\otimes dz^j \otimes dz^k + \barmu_{ij,\bark} dz^i \otimes dz^j \otimes dz^\bark  \nonumber\\
&& - \mu_{\bari\barj,k}dz^\bari\otimes dz^\barj \otimes dz^k - \mu_{\bari\barj,\bark}dz^\bari\otimes dz^\barj \otimes 
dz^\bark \nonumber\\
&& + \barmu_{ik,j} dz^i \otimes dz^j \otimes dz^k + \barmu_{ik,\barj} dz^i \otimes dz^\barj \otimes dz^k \nonumber\\
&& - \mu_{\bari\bark,j} dz^\bari \otimes dz^j \otimes dz^\bark - \mu_{\bari\bark,\barj} dz^\bari \otimes dz^\barj \otimes dz^\bark \nonumber\\
&& - \barmu_{jk,i} dz^i \otimes dz^j \otimes dz^k + \barmu_{jk,\bari} dz^\bari \otimes dz^j \otimes dz^k  \nonumber\\
&& - \mu_{\barj\bark,i} dz^i \otimes dz^\barj \otimes dz^\bark + \mu_{\barj\bark,\bari} dz^\bari \otimes dz^\barj \otimes dz^\bark. \nonumber
\end{eqnarray}
By Proposition 11 in \cite{La Fuente-Gravy 2016}, we see that
$ J((\nabla_X\mu)Y) - (\nabla_{JX}\mu)Y $
is symmetric in $X$ and $Y$. Taking $X = \barpartial_j$ and $Y = \barpartial_k$ we see that
$ 2\sqrt{-1}\nabla_\bark\mu^i{}_\barj $ is symmetric in $j$ and $k$. By lowering the upper index $i$,
we see $\nabla_\bark\mu_{\bari\barj} $ is symmetric in $j$ and $k$. By \eqref{symmetry}, $\mu_{\bari\barj} $
is symmetric in $i$ and $j$. Thus  $\nabla_\bark\mu_{\bari\barj} $ is totally symmetric in $i$, $j$ and $k$.
Similarly $\nabla_k\barmu_{ij} $ is totally symmetric in $i$, $j$ and $k$. It follows that the right hand side 
of \eqref{nablader1} is totally symmetric in $i$, $j$ and $k$. Further, using this symmetry, four terms in \eqref{nablader1}
cancel out, and we obtain \eqref{nablader}.
Instead of Proposition 11 in \cite{La Fuente-Gravy 2016}, one can use
the torsion-freeness of the Levi-Civita connection for the K\"ahler metrics $g_t$, see \cite{BCGRS}.
\end{proof}

We apply Lemma \ref{S^3} when the family $J_t$ is induced by the flow generated by a Hamiltonian
vector field $X_f$ of a smooth function $f$. Then by Lemma 2.3 in \cite{futaki06},
\begin{equation}\label{L_{X_f}J}
L_{X_f}J = 2\sqrt{-1}\nabla^{\prime\prime} X_f^\prime - 2\sqrt{-1}\nabla^{\prime} X_f^{\prime\prime}.
\end{equation}
Thus we have
$$\mu = \nabla^{\prime\prime} X_f^\prime.$$
But  $X_f$ is given by \eqref{X_f}, and thus
\begin{equation}\label{mu1}
\mu = -\sqrt{-1}\nabla^{\prime\prime} \grad^\prime f,
\end{equation}
and
\begin{equation}\label{L_{X_f}J2}
L_{X_f}J = 2\nabla^{\prime\prime}\grad^\prime f + 2\nabla^{\prime}\grad^{\prime\prime} f.
\end{equation}
Thus \eqref{nablader} becomes
\begin{eqnarray}\label{nablader2}
-\underline{\dot{\nabla}} &=& 
f_{ijk} dz^i\otimes dz^j \otimes dz^k + f_{\bari\barj\bark}dz^\bari\otimes dz^\barj \otimes dz^\bark \\
&& + f_{ij\bark} dz^i \otimes dz^j \otimes dz^\bark + f_{\bari\barj k}dz^\bari\otimes dz^\barj \otimes dz^k  \nonumber\\
&& + f_{ik\barj} dz^i \otimes dz^\barj \otimes dz^k + f_{\bari\bark j} dz^\bari \otimes dz^j \otimes dz^\bark \nonumber\\
&&  + f_{jk\bari} dz^\bari \otimes dz^j \otimes dz^k  
 + f_{\barj\bark i} dz^i \otimes dz^\barj \otimes dz^\bark.   \nonumber
 \end{eqnarray}
 Since the flow generated by $X_f$ induces the infinitesimal action $-L_{X_f}\nabla$ on $\mathcal E(M,\omega)$,
 the result of \eqref{nablader2} coincides with Lemma \ref{main lemma}.

%%%%%%%%%%%
%%%%%%%%%%%
%%%%%%%%%%%
\section{Cahen--Gutt extremal K\"ahler metrics}
We consider the following functional on $\mathcal J(M,\omega)$ which is similar to the Calabi functional
\cite{calabi85}. For $J\in \mathcal J(M,\omega)$ we consider the squared $L^2$-norm $\Phi$ of the moment map:
$\Phi(J) := \int_M \mu(\nabla^J)^2 \omega_m$. 
\begin{thm}\label{extremal}
A complex structure $J \in \mathcal J(M,\omega)$ is a critical point of $\Phi$ if 
$\grad^\prime \mu(\nabla^J)$ is a holomorphic vector field.
\end{thm}
\begin{proof}Let $J_t$ be a smooth curve in $\mathcal J(M,\omega)$ with $J_0 = J$. Then by Theorem \ref{CahenGutt},
we have
\begin{eqnarray}
\left.\frac{d}{dt}\right|_{t=0} \int_M \mu(\nabla^{J_t})^2\ \omega_m  &=&
2 \left.\frac{d}{dt}\right|_{t=0} \int_M \mu(\nabla^{J_t})\mu(\nabla^J)\ \omega_m\\
&=& 2\Omega^{\mathcal E}(L_{X_\mu}\nabla^J, \left.\frac{d}{dt}\right|_{t=0} \nabla^{J_t})\nonumber
\end{eqnarray}
where we have put $\mu = \mu(\nabla^J)$. By Proposition \ref{nondegenerate}, $L_{X_\mu}\nabla^J = 0$
if and only
if $L_{X_\mu} J = 0$, that is, if $\grad^\prime \mu(\nabla^J)$ is a holomorphic vector field.
\end{proof}
In this paper we call the K\"ahler metric $g = \omega(\cdot, J\cdot)$ a Cahen--Gutt extremal K\"ahler metric if $\grad^\prime \mu(\nabla^J)$ is a holomorphic vector field. 
Note that this is different from the critical metrics of Fox \cite{Fox2014}.

Now we follow the arguments of \cite{futaki07.1} and \cite{FO_reductive17} to obtain a Hessian formula of $\Phi$,
but we essentially follow the finite dimensional formal arguments of Wang \cite{Lijing06}. 
Here the Hessian is considered on the subspace 
of the tangent space of $\mathcal J(M,\omega)$ at $J$ consisting of the tangent vectors of the form
\begin{eqnarray}\label{tangent2}
\dot{J} &=& 4\Re \nabla^i\nabla_\barj\, \sqrt{-1}f\,\frac{\partial}{\partial z^i} \otimes d\overline{z^j}\\
&=& L_{JX_f} J \nonumber
\end{eqnarray}
for a real smooth function $f \in C^\infty(M)$.

Let $(M,\omega,J)$ be a compact K\"ahler manifold, and $\nabla$ the Levi-Civita connection.
We define Lichnerowicz operator
$\mathcal L$ of order 6 by
 \begin{eqnarray}\label{L}
 (h,\mathcal Lf)_{L^2} &=& 
 3(\nabla^{\prime}\nabla^{\prime\prime}\nabla^{\prime\prime}h, 
 \nabla^{\prime}\nabla^{\prime\prime}\nabla^{\prime\prime}f)_{L^2}
 - (\nabla^{\prime\prime}\nabla^{\prime\prime}\nabla^{\prime\prime}h, 
 \nabla^{\prime\prime}\nabla^{\prime\prime}\nabla^{\prime\prime}f)_{L^2}
 \nonumber\\
 &=& 3\int_M \overline{\nabla^\bari\nabla^j\nabla^k h}\,\nabla_i\nabla_\barj\nabla_\bark f\,\omega_m
 - \int_M \overline{\nabla^i\nabla^j\nabla^k h}\,\nabla_\bari\nabla_\barj\nabla_\bark f\ \,\omega_m
\end{eqnarray}
for any complex valued smooth functions $f$ and $h$. This is a self-adjoint elliptic differential operator of order 6. 
We further define the sixth order self-adjoint elliptic differential operator $\overline{\mathcal L} : C_\bfC^\infty(M) \to C_\bfC^\infty(M)$ by
 \begin{eqnarray}\label{barL}
 \overline{\mathcal L} u = \overline{\mathcal L\baru}.
\end{eqnarray}
\begin{lem}\label{derivative5}
If $\dot{J} = 4\Re \nabla^i\nabla_\barj\, \sqrt{-1}f\,\frac{\partial}{\partial z^i} \otimes d\overline{z^j}$ for a real valued smooth
function $f \in C^\infty(M)$, we have
 \begin{eqnarray*}
\left.\frac{d}{dt}\right|_{t=0}\, \mu(\nabla^{J(t)}) = \mathcal Lf + \overline{\mathcal L} f.
\end{eqnarray*}
\end{lem}
\begin{proof} 
Note that 
\begin{eqnarray}\label{dotJ}
\dot{J} =  2\sqrt{-1} \nabla_\barj (-\sqrt{-1}(\sqrt{-1}f)^i) - 2\sqrt{-1} \nabla_j (-\sqrt{-1}(\sqrt{-1}f)^\bari) 
\end{eqnarray}
For any real smooth function $h$ we obtain from Theorem \ref{CahenGutt}, \eqref{dotJ} and Lemma \ref{S^3} that
\begin{eqnarray*}
\left.\frac{d}{dt}\right|_{t=0}\, \int_M h\,\mu(\nabla^{J(t)})\,\omega_m 
= \Omega^{\mathcal E}(\underline{L_{X_h}\nabla^J}, \underline{A})
\end{eqnarray*}
where 
\begin{eqnarray*}
\underline{A} &=& \sqrt{-1}f_{ijk} dz^i\otimes dz^j \otimes dz^k - \sqrt{-1}f_{\bari\barj\bark}dz^\bari\otimes dz^\barj \otimes dz^\bark \\
&& + \sqrt{-1}f_{ij\bark} dz^i \otimes dz^j \otimes dz^\bark - \sqrt{-1} f_{\bari\barj k}dz^\bari\otimes dz^\barj \otimes dz^k  \nonumber\\
&& + \sqrt{-1}f_{ik\barj} dz^i \otimes dz^\barj \otimes dz^k - \sqrt{-1} f_{\bari\bark j} dz^\bari \otimes dz^j \otimes dz^\bark \nonumber\\
&&  + \sqrt{-1}f_{jk\bari} dz^\bari \otimes dz^j \otimes dz^k  
 - \sqrt{-1} f_{\barj\bark i} dz^i \otimes dz^\barj \otimes dz^\bark.   \nonumber
\end{eqnarray*}
Thus, we obtain 
\begin{eqnarray*}
&&\left.\frac{d}{dt}\right|_{t=0}\, \int_M h\,\mu(\nabla^{J(t)})\,\omega_m \\
&=&
2\Re ( -(\nabla^{\prime\prime}\nabla^{\prime\prime}\nabla^{\prime\prime}h, 
 \nabla^{\prime\prime}\nabla^{\prime\prime}\nabla^{\prime\prime}f)_{L^2}
 + 3(\nabla^{\prime}\nabla^{\prime\prime}\nabla^{\prime\prime}h, 
 \nabla^{\prime}\nabla^{\prime\prime}\nabla^{\prime\prime}f)_{L^2})\\
&& = (h,\mathcal Lf) + \overline{(h,\mathcal Lf)}\\
&& = (h,\mathcal Lf) + (h,\overline{\mathcal Lf}) = (h, \mathcal Lf + \overline{\mathcal L} f).
\end{eqnarray*}
This completes the proof.
\end{proof}
Recall that the Poisson bracket $\{h,f\}$ of two real smooth functions $h$ and $f$ is defined by
$$ \{h,f\} := X_h f$$
where $X_h$ is the Hamiltonian vector field of $h$, that is $i(X)\omega = dh$. For a K\"ahler form 
$\omega$, the Poisson bracket is expressed in terms local holomorphic coordinates as
\begin{eqnarray*}
\{h,f\} &=& \omega (X_h, J\grad f)\\
&=& dh(J\grad f)\\
&=& \sqrt{-1}f^\alpha h_\alpha - \sqrt{-1} f_\alpha h^\alpha
\end{eqnarray*}
\begin{lem}\label{equivariance}
For real valued smooth functions $f$ and $h$ in $C^\infty(M)$ we have
\begin{eqnarray*}
\Omega^{\mathcal E}(L_{X_h}\nabla^J,L_{X_f}\nabla^J) = - \int_M \{h,f\}\,\mu(\nabla^J)\omega_m.
\end{eqnarray*}
\end{lem}
\begin{proof}
Let $\sigma_t$ be the flow
generated by the Hamiltonian vector field of $h \in C^\infty(M)$. Since $\mu(\nabla^J)$ gives a 
$\mathrm{Ham}(M,\omega)$-equivariant moment map by Proposition \ref{CahenGutt} we have
\begin{equation}\label{equivariant2}
\int_M f\,\mu(\sigma_t(\nabla^J))\,\omega_m = \int_M f\circ\sigma_t^{-1}\,\mu(\nabla^J)\,\omega_m.
\end{equation}
Taking the time differential of $\sigma_t$ we obtain the lemma by (\ref{moment2}).
\end{proof}
\begin{lem}\label{Poisson}
For any smooth complex valued smooth function $f \in C^\infty(M)\otimes{\mathbf C}$ we have
\begin{equation*}
(\overline{\mathcal L} - \mathcal L)f = \sqrt{-1}\{f,\mu(\nabla^J)\} = f^\alpha \mu(\nabla^J)_\alpha - \mu(\nabla^J)^\alpha f_\alpha
\end{equation*}
where $f^\alpha = g^{\alpha\barbet}\partial f/\partial \overline{z^\beta}$ for local holomorphic coordinates 
$z^1, \cdots, z^m$.
\end{lem}
\begin{proof}
It is sufficient to prove when $f$ is real valued. For any real valued smooth function $h \in C^\infty(M)$ it
follows from Lemma \ref{equivariance} that
\begin{eqnarray*}
&&(h, \overline{\mathcal L} f - \mathcal L f)_{L^2} \\&&= 
\overline{3(\nabla^{\prime}\nabla^{\prime\prime}\nabla^{\prime\prime}h, 
 \nabla^{\prime}\nabla^{\prime\prime}\nabla^{\prime\prime}f)_{L^2}
-(\nabla^{\prime\prime}\nabla^{\prime\prime}\nabla^{\prime\prime}h, 
 \nabla^{\prime\prime}\nabla^{\prime\prime}\nabla^{\prime\prime}f)_{L^2}} \\
&& \qquad - 
(3(\nabla^{\prime}\nabla^{\prime\prime}\nabla^{\prime\prime}h, 
 \nabla^{\prime}\nabla^{\prime\prime}\nabla^{\prime\prime}f)_{L^2}
 -(\nabla^{\prime\prime}\nabla^{\prime\prime}\nabla^{\prime\prime}h, 
 \nabla^{\prime\prime}\nabla^{\prime\prime}\nabla^{\prime\prime}f)_{L^2}) \nonumber\\
  && = - \sqrt{-1} \Omega^{\mathcal E}(L_{X_h}\nabla^J,L_{X_f}\nabla^J)\\
 && = -\sqrt{-1} (\{f,h\}, \mu(\nabla^J))_{L^2}\\
 && = \sqrt{-1} (h, \{f,\mu(\nabla^J)\})_{L^2}\\
 && = (h, \mu(\nabla^J)_\alpha f^\alpha  - \mu(\nabla^J)^\alpha f_\alpha )_{L^2}.
\end{eqnarray*}
This completes the proof of Lemma \ref{Poisson}.
\end{proof}

\begin{lem}\label{first der}
If $f \in C^\infty(M)$ and $\dot{J} = 4\Re \nabla^i\nabla_\barj\, \sqrt{-1}f\,\frac{\partial}{\partial z^i} \otimes d\overline{z^j}$,
then
\begin{eqnarray}
\left.\frac{d}{dt}\right|_{t=0} \int_M \mu(\nabla^{J_t})^2 \omega_m &=& 
4(f, \mathcal L\mu(\nabla^J))_{L^2}\label{first der2}\\
&=& 4(f, \overline{\mathcal L} \mu(\nabla^J))_{L^2}.\nonumber
\end{eqnarray}
\end{lem}
\begin{proof} In this proof we write $\nabla$ instead of $\nabla^J$ for notational simplicity.
 We apply Theorem \ref{CahenGutt} and Lemma \ref{derivative5} to show
that the left hand side of (\ref{first der2}) is equal to 
\begin{eqnarray*}
&&2\Omega^{\mathcal E} (\underline{L_{X_{\mu(\nabla)}}\nabla}, \underline{\dot{\nabla}})\\
&=& 4\Re ( -(\nabla^{\prime\prime}\nabla^{\prime\prime}\nabla^{\prime\prime}\mu(\nabla), 
 \nabla^{\prime\prime}\nabla^{\prime\prime}\nabla^{\prime\prime}f)_{L^2}
 + 3(\nabla^{\prime}\nabla^{\prime\prime}\nabla^{\prime\prime}\mu(\nabla), 
 \nabla^{\prime}\nabla^{\prime\prime}\nabla^{\prime\prime}f)_{L^2})\\
&& = 2(\mu(\nabla),\mathcal Lf)_{L^2} + 2\overline{(\mu(\nabla),\mathcal Lf)_{L^2}}\\
&&  = 2(f, \mathcal L\mu(\nabla) + \overline{\mathcal L} \mu(\nabla))_{L^2}.
\end{eqnarray*}
But Lemma \ref{Poisson} implies 
$$ \overline{\mathcal L} \mu(\nabla) = \mathcal L\mu(\nabla).$$
Hence the left hand side of (\ref{first der2}) is equal to
$$ 4(f,\mathcal L\mu(\nabla))_{L^2} = 4(f,\overline{\mathcal L}  \mu(\nabla))_{L^2}.$$
\end{proof}

\begin{lem}\label{derL}
Suppose that $(\omega,J)$ gives a Cahen--Gutt extremal K\"ahler metric so that $J\grad\,\mu(\nabla^J)$ is a holomorphic vector field.
If $\dot{J} = 4\Re \nabla^i\nabla_\barj\, \sqrt{-1}f\,\frac{\partial}{\partial z^i} \otimes d\overline{z^j}$ for some real smooth function $f \in C^\infty(M)$ then we have
\begin{eqnarray*}
(\left.\frac{d}{dt}\right|_{t=0} \mathcal L)\mu(\nabla^J) &=& - \mathcal L(\mu(\nabla^J)^\alpha f_{\alpha} - f^\alpha \mu(\nabla^J)_\alpha) \\
&=& \mathcal L(\overline{\mathcal L}  - \mathcal L)f
\end{eqnarray*}
\end{lem}
\begin{proof}
First note that if 
$$ i(X_f) \omega = df$$
then 
\begin{equation}\label{JX_f}
L_{JX_f} \omega = 2i\partial\barpartial f.
\end{equation}
Let $\{\varphi_s\}$ be the flow generated by $- JX_f$. Let $S$ be a smooth function on $M$ such that $\grad\, S$ 
is a holomorphic vector field. We shall compute 
$ \left.\frac{d}{ds}\right|_{s=0} \mathcal L(\varphi_s J,\omega) S$, and apply to $S = \mu(\nabla^J)$, and obtain the conclusion of
Lemma \ref{derL}. Let $\{S_s\}$ be a family of smooth functions such that
$S_0 = S$, that
$$ \grad_s^\prime\,S_s = \grad^\prime\, S,$$
where $\grad_s$ denotes the gradient with respect to $\varphi_{-s}^\ast \omega$, and that
$$ \int_M S_s (\varphi_{-s}^\ast \omega)_m = \int_M S \omega_m. $$
This implies
\begin{equation}\label{S_s1}
\mathcal L(\varphi_sJ,\omega)\varphi_s^\ast S_s = \varphi_s^\ast(\mathcal L(J,\varphi_{-s}^\ast \omega) S_s) = 0.
\end{equation}
On the other hand, 
in general, if $\varphi_{-s}^\ast \omega = \omega + i\partial\barpartial h$ then $S_s = S + S^\alpha h_\alpha$.
Therefore, since $L_{JX_f} \omega = 2i\partial\barpartial f$ by (\ref{JX_f}) we have
\begin{equation}\label{S_s2}
S_s = S + 2s S^\alpha\,f_\alpha + O(s^2).
\end{equation}
Thus taking the derivative of (\ref{S_s1}), we obtain
\begin{equation}\label{derL2}
(\frac{d}{ds}|_{s=0}\mathcal L)S + \mathcal L(- (JX_f)S + 2S^\alpha f_\alpha) = 0.
\end{equation}
By an elementary computation we see
\begin{eqnarray*}
(JX_f)S &=& g(JX_f, \grad\,S) = \omega(X_f, \grad\, S) = df(\grad\, S)\\
&=& (\partial f + \barpartial f)(\grad^\prime S + \grad^{\prime\prime} S) = f_\alpha S^\alpha + f^\alpha S_\alpha.
\end{eqnarray*}
Thus, from (\ref{derL2}) and the above computation, we obtain
\begin{eqnarray*}
(\frac{d}{ds}|_{s=0}\mathcal L)S &=& \mathcal L( (f_\alpha S^\alpha + f^\alpha S_\alpha) - 2S^\alpha f_\alpha) = 0\\
&=&  \mathcal L (f^\alpha S_{\alpha} - f_\alpha S^\alpha)\\
&=& \mathcal L(\overline{\mathcal L}  - \mathcal L) f.
\end{eqnarray*}
This completes the proof of Lemma \ref{derL}.
\end{proof}

To express the Hessian formula of $\Phi$, we consider its restriction to the subspace consisting of tangent vectors
of the form
$\dot{J} = 4\Re \nabla^i\nabla_\barj\, \sqrt{-1}f\,\frac{\partial}{\partial z^i} \otimes d\overline{z^j} = L_{JX_f} J$ 
for a real smooth function $f \in C^\infty(M)$.

\begin{thm}\label{Hessian}
Suppose that $J$ gives a Cahen--Gutt extremal K\"ahler metric so that $J$ is a critical point of $\Phi$.
Let $f$ and $h$ be real smooth functions in $C^\infty(M)$. Then the Hessian $\mathrm{Hess}(\Phi)_J$ at
$J$ is given by
$$\mathrm{Hess}(\Phi)_J(L_{JX_f}J, L_{JX_h}J)
= 8(f, \mathcal L\overline{\mathcal L}  h)_{L^2} = 8(f,\overline{\mathcal L}  \mathcal L h)_{L^2}. $$
In particular, at any point $J$ giving a Cahen--Gutt extremal K\"ahler metric, we have $\mathcal L\overline{\mathcal L}  = \overline{\mathcal L}  \mathcal L$ on the space $C_\bfC^\infty(M)$ of smooth complex valued functions since $\mathcal L\overline{\mathcal L} $ and $\overline{\mathcal L}  \mathcal L$ are both
$\bfC$-linear.
\end{thm}
\begin{proof}
Suppose $\dot{J} = 4\Re \nabla^i\nabla_\barj\, \sqrt{-1}h\,\frac{\partial}{\partial z^i} \otimes d\overline{z^j}$. Then by Lemma \ref{first der}, Lemma \ref{derL} 
and Lemma \ref{derivative5} we obtain
\begin{eqnarray*}
\mathrm{Hess}(\Phi)_J(L_{JX_f}J, L_{JX_h}J) &=& 
\left.\frac{d}{dt}\right|_{t=0} 4(f,\mathcal L\mu(\nabla^J))_{L^2}\\
&=& 4(f, (\left.\frac{d}{dt}\right|_{t=0} \mathcal L)\mu(\nabla^{J}) + \mathcal L\left.\frac{d}{dt}\right|_{t=0} \mu(\nabla^{J_t}))_{L^2}\\
&=& 4(f, \mathcal L(\overline{\mathcal L}  - \mathcal L)h + \mathcal L(\mathcal L + \overline{\mathcal L} )h)_{L^2}\\
&=& 8(f,\mathcal L\overline{\mathcal L}  h)_{L^2}.
\end{eqnarray*}
Similarly, we obtain
\begin{eqnarray*}
\mathrm{Hess}(\Phi)_J(L_{JX_f}J, L_{JX_h}J) &=& 
\left.\frac{d}{dt}\right|_{t=0} 4(f,\overline{\mathcal L}  \mu(\nabla^J))_{L^2}\\
&=& 4(f, (\left.\frac{d}{dt}\right|_{t=0} \overline{\mathcal L} )\mu(\nabla) + \overline{\mathcal L} \left.\frac{d}{dt}\right|_{t=0} \mu(\nabla^{J_t}))_{L^2}\\
&=& 4(f, \overline{\mathcal L} (\mathcal L - \overline{\mathcal L} )h + \overline{\mathcal L} (\mathcal L + \overline{\mathcal L} )h)_{L^2}\\
&=& 8(f,\overline{\mathcal L}  \mathcal L h)_{L^2}.
\end{eqnarray*}
This completes the proof of Theorem \ref{Hessian}.
\end{proof}
Let 
$\mathfrak g = \mathfrak g_0 + \sum_{\lambda \ne 0} \mathfrak g_\lambda$
be the eigenspace decomposition of $\mathfrak g$ with respect to the action of
$\mathrm{ad}(\grad^\prime \mu(\nabla))$, and $\mathfrak e = \mathfrak e_0 + \sum_{\lambda \ne 0} \mathfrak e_\lambda$ be the corresponding decomposition of the space of complex valued 
potentials functions. Thus $\dim \mathfrak e = \dim \mathfrak g +1$ because of the constant functions.
\begin{thm}\label{Calabi1}
Let $g = \omega J$ is a Cahen--Gutt extremal K\"ahler metric on a compact K\"ahler manifold.
Then we have the following:
\begin{itemize}
\item[(a)] The space $\mathfrak e$ 
of potential functions  of the reduced Lie algebra $\mathfrak g$ is included in $\ker \mathcal L$. 
\item[(b)] $\overline{\mathcal L} $ maps $\mathfrak e$ into itself 
and coincides with with the Poisson bracket with $\mu(\nabla)$. In particular, the eigenspace 
decomposition of $\overline{\mathcal L} : \mathfrak e \to \mathfrak e$ coincides with 
the decomposition $\mathfrak e = \mathfrak e_0 + \sum_{\lambda \ne 0} \mathfrak e_\lambda$.
\item[(c)] $\mathcal L$ and $\overline{\mathcal L}$ coincide when restricted to $\mathfrak e_0$,
and are real operators on $\mathfrak e_0$.
\end{itemize}
\end{thm}
\begin{proof}
By the definition \eqref{L} of $\mathcal L$ it is clear that $\mathfrak e$ is included 
in $\ker \mathcal L$. Thus (a) follows. 
By Lemma \ref{Poisson}, we have for $f \in \mathfrak e$ 
\begin{eqnarray}
\overline{\mathcal L}  f &=& (\overline{\mathcal L}  - \mathcal L) f\nonumber\\
&=& \mu(\nabla)^\alpha f_\alpha - f^\alpha \mu(\nabla)_\alpha, \label{Poisson2}
\end{eqnarray}
and the right hand side is the Poisson bracket $\{\mu(\nabla),f\}$ 
and belongs to $\mathfrak e$. Further we have
 $$ \grad^\prime \{\mu(\nabla),f\} = [\grad^\prime \mu(\nabla), \grad^\prime f].$$
Thus the eigenspace decompositions of $\overline{\mathcal L}$ coincides with 
$\mathfrak e = \mathfrak e_0 + \sum_{\lambda \ne 0} \mathfrak e_\lambda$.
This proves (b). 
The equality \eqref{Poisson2} shows $\mathcal L = \overline{\mathcal L}$ on $\mathfrak e_0$.
This proves (c).
\end{proof}

\begin{lem}\label{Ric} Let $M$ be a compact K\"ahler manifold.
If $g = \omega J$ is a Cahen--Gutt extremal K\"ahler metric with non-negative Ricci curvature 
then $\mathfrak e = \ker \mathcal L$.
\end{lem}
\begin{proof} By (a) of Theorem \ref{Calabi1} we have only to 
show $\ker \mathcal L \subset \mathfrak e$. 
Since
\begin{eqnarray*}
- \nabla^i\nabla_i\nabla_\barj\nabla_\bark f + \nabla^\bari\nabla_\bari\nabla_\barj\nabla_\bark f
= 2 R^\barl{}_\barj \nabla_\bark\nabla_\barl f
\end{eqnarray*}
for any complex valued smooth function $f$
we see using \eqref{L} that if the Ricci curvature is non-negative
\begin{equation}\label{non-negative}
(f,\mathcal L f)_{L^2} \ge 2(\nabla^{\prime}\nabla^{\prime\prime}\nabla^{\prime\prime}f, 
 \nabla^{\prime}\nabla^{\prime\prime}\nabla^{\prime\prime}f)_{L^2}.
 \end{equation}
Thus $\mathcal L f = 0$ implies 
$\nabla^\prime\nabla^{\prime\prime}\nabla^{\prime\prime}f = 0$. By integration by parts, this shows
$$ \int_M |\nabla^{\prime\prime}\nabla^{\prime\prime}f|^2 \omega_m = 0.$$
This implies $\grad^\prime f$ is holomorphic. This shows $\ker \mathcal L \subset \mathfrak e$.
\end{proof}
\begin{rem}\label{Prop 17}The condition of non-negative Ricci curvature coincides with the non-degenracy condition of $lv^\ast\Omega^{\mathcal E}$ due to La Fuente-Gravy, Proposition 17 in \cite{La Fuente-Gravy 2016}.
\end{rem}

Now we show a Cahen--Gutt version of Calabi's theorem \cite{calabi85} for extremal K\"ahler metrics. Before stating it we remark that it is well-known that for a Killing vector field $X$ on a compact K\"ahler manifold $M$, the complex vector field $X - iJX$ is a holomorphic vector field. We identify the real Lie algebrs $\mathfrak i(M)$ of all Killing vector fields on $M$ with the real Lie subalgebra of the complex Lie algebra of all holomorphic vector fields by the identification $X \mapsto X - iJX$. 
\begin{thm}\label{Calabi2} Let $M$ be a compact K\"ahler manifold.
If $g = \omega J$ is a Cahen--Gutt extremal K\"ahler metric with non-negative Ricci curvature then the reduced Lie algebra $\mathfrak g$ of holomorphic vector fields
has the following structure:
\begin{itemize}
\item[(a)] $\grad^\prime \mu(\nabla) = g^{i\barj}\frac{\partial \mu(\nabla)}{\partial \overline{z^j}}\frac{\partial}{\partial z^i}$ is in the center of $\mathfrak g_0$.
\item[(b)] $\mathfrak g = \mathfrak g_0 + \sum_{\lambda > 0} \mathfrak g_\lambda$ where 
$\mathfrak g_\lambda$ is the $\lambda$-eigenspace of $\mathrm{ad}(\grad^\prime \mu(\nabla))$. Moreover, we have
$[\mathfrak g_\lambda,\mathfrak g_\mu] \subset \mathfrak g_{\lambda + \mu}$. 

\item[(c)] $\mathfrak g_0$ is isomorphic to $\mathfrak i(M)\otimes \bfC$, and is the maximal reductive subalgebra of $\mathfrak g$ where $\mathfrak i(M)$ denotes the real Lie algebra of all Killing vector fields. In particular $\mathfrak g_0$ is reductive. Further, the identity component of the isometry group is a maximal compact subgroup of the identity component of the group of all biholomorphisms of $M$. 
\end{itemize}
\end{thm}
\begin{proof}Since $\mathfrak g_0$ is the $0$-eigenspace with respect to the action of $\mathrm{ad}(\grad^\prime \mu(\nabla))$, $\grad^\prime \mu(\nabla)$ is in the center of $\mathfrak g_0$. This proves (a). 
By \eqref{non-negative} $\mathcal L$ is a non-negative operator the non-zero eigenvalues of $\mathcal L$ are positive. Taking the complex-conjugate the same is true for $\overline{\mathcal L}$. Since by Theorem \ref{Calabi1}, (b), the non-zero eigenvalues of $\mathcal L$  coincide with those of $\mathrm{ad}(\grad^\prime \mu(\nabla))$ we have $\lambda > 0$ for all non-zero $\lambda$'s. This proves (b). 
From Theorem \ref{Calabi1}, (c), we see $f \in \mathfrak e_0$ satisfies both $\mathcal L f = 0$ and $\mathcal L\barf = 0$. Since $\overline{\mathcal L}  f = 0$ is equivalent to $\mathcal L\barf = 0$, this implies $\mathcal L\Re f = 0$ and $\mathcal L\Im f = 0$. Hence by Lemma \ref{Ric}, both the real and imaginary part of $f$ are potential functions of holomorphic vector fields. In general if $\grad^\prime f$ is a holomorphic vector field for a real smooth function $f$ then $J\grad\, f$ is a Killing vector field. It is also well-known that for a Killing vector field $X$ on a compact K\"ahler manifold, $X - iJX$ is a holomorphic vector field. Hence we obtain $\mathfrak g_0 = \mathfrak i(M)\otimes \bfC$. In particular $\mathfrak g_0$ is reductive. To show that this is a maximal reductive Lie algebra, suppose we have a reductive Lie subalgebra $\mathfrak l$ containing $\mathfrak g_0$. If $\mu(\nabla)$ is constant we have $\mathfrak g = \mathfrak g_0$ and thus $\mathfrak g_0$ is maximal. Thus we may assume $\mu(\nabla)$ is not constant. Let $X$ be an element of $\mathfrak l$ in the form $X = X_0 + \sum_\lambda X_\lambda$ where $X_0 \in \mathfrak g_0$ and $X_\lambda \in \mathfrak g_\lambda$. Since
$$ Ad (\exp(t\,\grad^\prime \mu(\nabla))X = X_0 + \sum_{\lambda > 0} e^{\lambda t} X_\lambda \in \mathfrak l$$
for any $t \in \bfR$, by considering this for many values of $t$ we have $X_\lambda \in \mathfrak l$ for each $\lambda$. If $X_\lambda \ne 0$ for some $\lambda > 0$ then $\grad^\prime \mu(\nabla) + X_\lambda$ generates a solvable Lie subalgebra. This contradicts the reductiveness of $\mathfrak l$. Thus  $X_\lambda = 0$ for all $\lambda > 0$. Thus $\mathfrak g_0$ is a maximal reductive Lie subalgebra. To show the last statement let $K$ be a compact connected Lie subgroup of the group of all biholomorphisms of $M$ including the identity component of the group of all isometries of $M$, and $\mathfrak k$ be its real Lie algebra. Since $\mathfrak k \otimes \bfC$ is reductive and $\mathfrak g_0$ is a maximal reductive Lie subalgebra we must have $\mathfrak k \otimes \bfC = \mathfrak g_0 = \mathfrak i(M) \otimes \bfC$. Thus any element of $\mathfrak k$ can be written in the form $ J\grad\, f + \grad\, h$ for some real potential functions $f$ and $h$ where $J\grad\,f$ and $J\grad\,h$ are Killing vector fields (by using the identification remarked before the statement of Theorem \ref{Calabi2}). Since $\mathfrak i(M)$ is a Lie subalgebra of $\mathfrak k$ we have $\grad\,h \in \mathfrak k$. If $\grad\,h$ is non-zero it generates a non-compact group since $h$ has at least two critical points and $K$ can not be compact. Thus $\grad\,h$ has to be zero and $\mathfrak k = \mathfrak i(M)$. This proves (c).
\end{proof}
%By Theorem \ref{Hessian}, 
%$\mathcal L\overline{\mathcal L}  = \overline{\mathcal L}  \mathcal L$ on $C^\infty_\bfC(M)$. 
%Therefore $\overline{\mathcal L} $ maps $\ker \mathcal L$ to $\ker \mathcal L$, 
%and we have the direct sum decomposition 
%$$ \ker \mathcal L = \sum_\lambda E_\lambda $$
%into the eigenspaces $E_\lambda$ of $\overline{\mathcal L} $. Further by Lemma \ref{Poisson}, we have for $f \in E_\lambda$ 
%\begin{eqnarray*}
%\lambda f &=& \overline{\mathcal L}  f \\
%&=& (\overline{\mathcal L}  - \mathcal L) f\\
%&=& \mu(\nabla)^\alpha f_\alpha - f^\alpha \mu(\nabla)_\alpha.
%\end{eqnarray*}
%This shows
%$$ [\grad^\prime \mu(\nabla), \grad^\prime f] = \lambda \grad^\prime f.$$
%Thus we have $E_\lambda = \mathfrak g_\lambda$. For $\lambda = 0$ we have
%$$ \mathfrak g_0 = E_0 = \ker \mathcal L \cap \ker \overline{\mathcal L} .$$

\begin{proof}[Proof of Theorem \ref{reductiveness}]
As shown by La Fuente-Gravy \cite{La Fuente-Gravy 2016}, if $lv(J)$ gives rise to 
closed Fedosov star product then $\mu(\nabla)$ is constant and thus $\mathfrak g = \mathfrak g_0$. It follows from Theorem \ref{Calabi2}, (c), 
that $\mathfrak g$ is reductive. This completes the proof.
\end{proof}

\begin{ex}
Let $M$ be a one point blow-up of the complex projective plane $\bfC\bfP^2$. Since $M$ is simply connected the reduced Lie algebra of holomorphic vector fields coincides with the Lie algebra of all holomorphic vector fields. The corresponding Lie group is the group of all biholomorphic automorphisms.
Any automorphism of $M$ leaves the exceptional divisor invariant, and thus descends to an automorpism
fixing the point where the blow-up is performed. Thus the reduced Lie algebra is of the form
$$ \left.\left\{\left(\begin{array}{ccc} \ast & \ast &\ast \\ 0& \ast &\ast \\ 0 &\ast &\ast \end{array}\right) \right\}\right/\mathrm{center}
$$
which is not reductive. Thus any K\"ahler metric with non-negative Ricci curvature on $M$ does not give closed Fedosov star product.
This example is also the simplest example of a compact K\"ahler manifold with no cscK metric.
\end{ex}

%%%%%%

%%%%%%%

%%%%%%%%%%%%%%%%%%%%%%%

\end{document}